\documentclass[12pt,a4paper]{amsart}
\usepackage{amsmath}
\usepackage{amsthm}
\usepackage{amssymb}
\usepackage{mathrsfs}
\usepackage[british]{babel}
\usepackage{ot1patch}
\usepackage{graphicx}
\usepackage{verbatim}
\usepackage{subcaption}
\usepackage[leftcaption]{sidecap}

\newtheorem{thm}{Theorem}[section]
\newtheorem{lem}[thm]{Lemma}
\newtheorem{cor}[thm]{Corollary}
\newtheorem{prop}[thm]{Proposition}
\newtheorem{que}[thm]{Question}

\theoremstyle{remark}
\newtheorem{rem}[thm]{Remark}
\newtheorem*{rem*}{Remark}

\theoremstyle{definition}
\newtheorem{dfn}[thm]{Definition}
\newtheorem{ex}[thm]{Example}
\numberwithin{equation}{section}

\graphicspath{ {Graphics/} }

\usepackage[giveninits=true]{biblatex}
\title{Medial Axis in Pseudo-Euclidean Spaces}
\author{Adam Białożyt}
\address{Jagiellonian University, Faculty of Mathematics and Computer Science, Institute of Mathematics, \L ojasiewicza 6, 30-348 Krak\'ow, Poland}\email{adam.bialozyt@uj.edu.pl} 
\keywords{Medial Axis, Skeleton, Radius of Curvature, Dimension, Tangent Cone, O-minimal Geometry}
\subjclass[2010]{32B20, 53A07, 54F45}
\begin{document}
\begin{abstract}
    In the paper we will investigate the notion of medial axis for psuedo-Euclidean spaces. For most of the article we follow the path of Birbrair and Denkowski \cite{BirbrairDenkowski} checking its feasibility in the new context. 
\end{abstract}

\maketitle
\section{Preliminaries}
\begin{dfn}
We call the pair $(\mathbb{R}^n,Q)$, where $Q$  is a nondegenerated quadratic form, \textit{a pseudo-Euclidean space}. 
\end{dfn}
For any pseudo-Euclidean space $(\mathbb{R}^n,Q)$ there exists exactly one symmetric matrix $A\in M(n)$ such that 
$$Q(x)=x^TA^Tx.$$
For such a matrix and any two vectors $x,y\in\mathbb{R}^n$ we introduce their pseudo-scalar product $\langle\langle x,y\rangle\rangle:=x^TAy$. It is easy to check that the pseudoscalar product is symmetrical, bilinear and nongenerated (meaning $\exists x\forall y \langle\langle x,y\rangle\rangle=0 \Rightarrow x=0 $). From the general theory of quadratic forms it is known that we can fix a orthonormal basis of $\mathbb{R}^n$ in such a way that $A=diag(\lambda_1,\ldots,\lambda_n),\,\lambda_i\neq 0,\, \lambda_i\geq\lambda_{i+1}$. Then $l:=\max\lbrace i\mid\lambda_i>0\rbrace$ is called the index of the pseudo-Eculidean space, and a pair $(l,p):=(l,n-l)$ is called its signature. In place of $\mathbb{R}^n$ it customed to write in such a setting $\mathbb{R}_{(l,p)}$. This notation makes sense as all $\mathbb{R}_{(l,p)}$ are lineary isometric. In this spirit it will not affect the generalicity of our study to assume that all $\lambda_i$ have absolute value equal one. In such setting we denote $$\|x\|_l=(\sum_{i=1}^l x_i^2)^{1/2}\text{ and }\|x\|_p=(\sum_{i=l+1}^p x_i^2)^{1/2}.$$

For any $x\in\mathbb{R}_{(l,p)}$ the value $Q(x)$ is a generalisation of the vector $x$ norm squared. Since we do not assume $Q$ to be positive definite, $Q(x)$ can admit any real value. Henceforth we distinguish in the pseudo-Eulidean spaces the vectors that are 
\begin{itemize}
    \item space like $Q(x)>0$
    \item light like $Q(x)<0$
    \item isotropic $Q(x)=0$
\end{itemize}
Analogously the two dimensional subspace $V$ of $\mathbb{R}_{(l,p)}$ is called 
\begin{itemize}
    \item Euclidean if $Q|V$ is definite (either positive or negative).
    \item pseudo-Euclidean if $Q|V$ is non-singular indefininite.
    \item isotropic in the remaining cases.
\end{itemize}

Since the quadratic form $Q$ is not convex the triangle inequality does not hold in general. However for any two $x,y\in\mathbb{R}_{(l,p)}$ there is
$$|Q(x+y)|\leq |Q(x)|+|Q(y)|,\text{ if } x,y\text{ span an Euclidean subspace of }\mathbb{R}_{(l,p)}$$
$$|Q(x+y)|\geq |Q(x)|+|Q(y)|,\text{ if } x,y\text{ span a pseudo-Euclidean subspace of }\mathbb{R}_{(l,p)}.$$

For a real number $r$ and $x\in\mathbb{R}_{(l,p)}$ we can define also the equivalents of open and closed balls
$$B(x,r):=\lbrace y\in\mathbb{R}_{(l,p)}\mid Q(x-y)< r\rbrace,\;\overline{B}(x,r):=\lbrace y\in\mathbb{R}_{(l,p)}\mid Q(x-y)\leq r\rbrace.$$
Mind that we will use the sets $B(x,r)$ and $\overline{B}(x,r)$ only as a tool in analysis of the mutual position of points in pseudo-Euclidean spaces. When it comes to the typically topological and metric properties like boundedness or compactness we will still mean the properties inherited from the natural topology of $\mathbb{R}^n$  
Mind also that for the simplicity of the definition of $B(x,r)$ we take $r$ to be equal the usual radius squared.

\section{Medial Axis}
For the whole article we will assume that $X\subset \mathbb{R}_{(l,p)}$ is closed in the Euclidean topology.
Then for a point $a\in\mathbb{R}_{(l,p)}$ we define the squared distance function
$$\rho(a,X):=\inf\lbrace Q(x-a)\mid x\in X\rbrace$$
and a multifunction of the closest points
$$m(a)=\lbrace x\in X\mid  Q(x-a)=\rho(a,X)\rbrace.$$
Note here that $m$ has also another equivalent and extremely useful definition
$$m(a)=\lbrace x\in X\mid \forall p\in X,\, Q(a-x)\leq Q(a-p)\rbrace.$$

Contrary to the Euclidean case, in general the function $\rho$ is not bounded from below, and not all points $a$ bring nonempty or compact $m(a)$. Indeed for the set $X=\lbrace x_2=\sqrt{x_1^2+1}\rbrace\subset\mathbb{R}_{(1,1)}$ we have $m(0)=X$ and $m(a)=\emptyset$ for any other $a$ on $x_1-$axis.

Therefore besides the usual notion of \textit{the medial axis}
$$M_X:=\lbrace a\in\mathbb{R}_{(l,p)}\mid \#m(a)>1\rbrace$$

a new potentially interesting set appears --\textit{ the vacant axis} of $X$
$$W_X:=\lbrace a\in\mathbb{R}_{(l,p)}\mid m(a)=\emptyset \rbrace.$$

Moreover, it is not all certain that the multifunction $m$ seen as a relation is reflexive for a general set $X$ (consider points of $X=0\times[0,1]$ in $\mathbb{R}_{(1,1)}$). To enforce reflexivity, we will assume $X$ to be \textit{acausal}\footnote{The name comes from the Lorentzian geometry} meaning for any two distinguished points $x,y\in X$ we assume to have $Q(x-y)>0$. It is plain to see that in such a case $\pi_l:\mathbb{R}_{(l,p)}\to\mathbb{R}^l$, the natural projection on the first $l$ coordinates, must be injective on $X$. If we consider now $\pi_p:\mathbb{R}_{(l,p)}\to\mathbb{R}^p$ to be the natural projection on the last $p$ coordinates, and the composition $\pi_p\circ(\pi_l|X)^{-1}:\pi_l(X)\to\mathbb{R}^p$, it is easy to check that $\pi_p\circ(\pi_l|X)^{-1}$ is $1-$Lipschitz for acausal $X$. Based on that we will call $X$ to be \textit{$L-$pseudo-Lipschitz} if $\pi_p\circ(\pi_l|X)^{-1}$ is $L-$Lipschitz.  It translates to the condition $$L\|x-y\|_l\geq\|x-y\|_p\text{, for all }x,y\in X.$$

\section{multifunctions $m$ and $\mathcal{N}$}

Before delving any deeper into the analysis of medial axis in pseudo-Euclidean spaces, let us justify the choice of restriction we imposed on sets of our interest.
\begin{thm}\label{medial axis of acausal set}
For any point $x$ of an acausal set $X$, $m(x)=\{x\}$ and the intersection $M_X\cap X$ is empty.
\end{thm}
\begin{proof}
Taking $x\in X$ we obviously have $Q(x-x)=0$, whereas $$Q(x-y)=\|x-y\|_l^2- \|x-y\|_p^2=\|\pi_l(x)-\pi_l(y)\|^2-\|\pi_p(x)-\pi_p(y) \|^2>0$$ for any point $y\in X$ distinct from $x$. Hence $m(x)=\lbrace x\rbrace$ and $x$ is not a point of $M_X$.
\end{proof}

Thanks to the Theorem~\ref{medial axis of acausal set}, we can define \textit{the normal sets} of an acausal set just like in the Euclidean case.

$$\mathcal{N}(a):=\lbrace x\in\mathbb{R}_{(l,p)}\mid a\in m(x)\rbrace = \lbrace x\in\mathbb{R}_{(l,p)}\mid Q(x-a)=\rho(x,X)\rbrace$$

$$\mathcal{N}'(a):=\lbrace x\in \mathbb{R}_{(l,p)}\mid m(a)=\lbrace a\rbrace\rbrace.$$

\begin{prop}\label{properties of N}
For any $a$ in an acausal definable subset $X$ of $\mathbb{R}_{(l,p)}$ there is
\begin{enumerate}
    \item $a\in\mathcal{N}'(a)\subset\mathcal{N}(a)$,
    \item $\mathcal{N}(a)$ is close convex and definable
    \item $\mathcal{N}(a)-a\subset \mathcal{N}_aX:=\lbrace w\in\mathbb{R}_{(l,p)}\mid \forall v\in C_aX:\; \langle\langle w,v\rangle\rangle\leq 0\rbrace$.
    \item $x\in\mathcal{N}'(a)\Rightarrow [a,x]\subset \mathcal{N}'(a)$,  $x\in\mathcal{N}(a)\backslash\lbrace a\rbrace \Rightarrow [a,x)\subset\mathcal{N}'(a)$
    \item $\limsup_{X\in b\to a}\mathcal{N}(b)\subset\mathcal{N}(a)$
    \item $\mathcal{N}'(a)$ is convex and definable.
\end{enumerate}

\end{prop}
\begin{proof}
$(1)$ $a\in\mathcal{N}'(a)$ is given by the previous theorem, the inclusion is trivial.

$(2)$ The set is definable due to the description $$\mathcal{N}(a)=\lbrace x\in \mathbb{R}_{(l,p)}\mid \forall b\in X:\; Q(x-a)\leq Q(x-b)\rbrace$$ since $Q$ and $X$ are definable. If we move the quantifier in front of the set braces we obtain  $$\mathcal{N}(a)=\bigcap_{b\in X} \lbrace x\in \mathbb{R}_{(l,p)}\mid Q(x-a)\leq Q(x-b)\rbrace$$
which gives closedness and convexity of $\mathcal{N}(a)$ as for any $b\in X$ the set $\lbrace x\in \mathbb{R}_{(l,p)}\mid Q(x-a)\leq Q(x-b)\rbrace$ is a closed halfspace.

$(3)$ For any $b\in X$, the set $\mathcal{N}(a)$ is a subset of $$\lbrace x\in \mathbb{R}_{(l,p)}\mid Q(x-a)\leq Q(x-b)\rbrace=\lbrace x+a\mid Q(x+a-a)\leq Q(x+a-b)\rbrace=$$
$$=\lbrace x\mid Q(x)\leq Q(x+(a-b))\rbrace+a=\lbrace x\mid  0\leq 2\langle\langle x,a-b\rangle\rangle+Q(a-b)\rbrace+a.$$
Now, taking $v\in C_aX$ and sequences $b_\nu\to a,\,l_\nu\to 0^+$ such that $l_\nu(b_\nu-a)\to v$ we have
$$\mathcal{N}(a)\subset\lbrace x\mid   2\langle\langle x,l_\nu(b_\nu-a)\rangle\rangle\leq\langle\langle a-b_\nu,l_\nu(b_\nu-a)\rangle\rangle\rbrace+a$$
and by passing to the limit we obtain
$$\mathcal{N}(a)\subset\lbrace x\mid   2\langle\langle x,v\rangle\rangle\leq 0\rbrace+a,\, \forall v\in C_aX$$

$(4)$ For any triple of points $a,b\in X, x\in \mathcal{N}(a)$ let us define a polynomial $$R(t):=Q(tx+(1-t)a-b)-Q(tx+(1-t)a-a).$$ After simplification, $R$ reduces to $R(t)=2t\langle\langle x-a,a-b\rangle\rangle+Q(a-b)$, thus it is a degree $1$ polynomial and $R(0)>0$, $R(1)\geq 0$. Therefore, $R|(0,1)$ is positive and consequently $m(tx+(1-t)a)=\lbrace a\rbrace$ thus $[a,x)\subset \mathcal{N}'(a)$.

$(5)$ Take any sequence $X\ni b_\nu\to a$ and $x_\nu\in\mathcal{N}(b_\nu)\to x$. Then $b_\nu\in m(x_\nu)$ thus $\forall b\in X:\; Q(b_\nu-x_\nu)\leq Q(b-x_\nu).$
After passing to the limit we obtain $Q(a-x)\leq Q(b-x)$ for all $b\in X$ thus $a\in m(b)$.

$(6)$ Follows the same reasoning as $(2).$
\end{proof}

\begin{prop}
For any point $a\in X$, there is $\mathcal{N}(a)=\overline{\mathcal{N}'(a)}$
\end{prop}
\begin{proof}

$\mathcal{N}(a)$ is closed and 
$\mathcal{N}'(a)\subset\mathcal{N}(a)$ thus $\overline{\mathcal{N}'(a)}\subset \mathcal{N}(a)$.
On the other hand for any $x\in\mathcal{N}(a)\backslash\{a\}$ there is $[a,x)\subset\mathcal{N}'(a)$ thus $x\in \overline{\mathcal{N}'(a)}$.
\end{proof}

\begin{thm} For any closed $X$
$$M_X=\bigcup \mathcal{N}(a)\backslash\mathcal{N}'(a)=\bigcup \mathcal{N}(a)\backslash\bigcup \mathcal{N}'(a)=\mathbb{R}_{(l,p)}\backslash \bigcup \mathcal{N}'(a)\cup W_X.$$
\end{thm}
\begin{proof}
The proof is basically the same as in \cite{BirbrairDenkowski}
\end{proof}

Recall that, we call $X$ to be \textit{$L-$pseudo-Lipschitz} if $\pi_p\circ(\pi_l|X)^{-1}$ (where $\pi_l$ and $\pi_p$ are natural projections on the first $l$ and last $p$ coordinates respectively) is $L-$Lipschitz. We have the following result concerning the  multifunction $m$.

\begin{thm}\label{continuity of m}
Let $X$ be a closed $L-$pseudo-Lipschitz definable subset of $\mathbb{R}_{(l,p)}$ with $L< 1$, then $W_X=\emptyset$. Moreover, the multifunction $m$ has compact values, is locally bounded and upper-semicontinuous, meaning
$\limsup_{D\ni x\to a}m(x)=m(a)$ for any dense subset $D$.
\end{thm}
\begin{proof}
Since $X$ is $L-$pseudo-Lipschitz, for any $b\in X$ there is $$X\subset B:=\lbrace x\in\mathbb{R}_{(l,p)}\mid \,L\|b-x\|_l\geq \|b-x\|_p\rbrace.$$
In such a case for $a\in \mathbb{R}_{(l,p)}$ and any sequence $X\ni x_\nu\to \infty$ the values of $Q(x_\nu-a)$ diverge to infinity, since the cones at infinity of $\lbrace x\in \mathbb{R}_{(l,p)}\mid Q(x-a)\leq const.\rbrace$ and $B$ are disjoint. It means that $\lbrace x\in X\mid Q(x-a)\leq \rho(a)+\varepsilon\rbrace$ is bounded, and therefore compact and from any sequence realising the infimum $\rho(a)$ we can pick a subsequence convergent in $X$. Since $Q$ is continuous the limiting point in $X$ realises the value $\rho(a)$. Moreover, $$m(a)=\bigcap_{\varepsilon>0}\lbrace x\in X\mid Q(x-a)\leq \rho(a)+\varepsilon\rbrace$$
thus $m(a)$ is compact, as it is closed and bounded.

Regarding local boundedness of $m$, let us star by remarking that $\rho$ is upppersemi-continuous, as it is an infimum of a collection of continuous functions. It means that for any $\varepsilon>0$ we can find $U-$a neighbourhood of $a$ such that for all $\tilde{a}\in U$ there is $\rho(\tilde{a})<\rho(a)+\varepsilon$. Taking a union of $\lbrace x\in X\mid Q(x-\tilde{a})\leq \rho(a)+\varepsilon\rbrace$ over $U$ we obtain a set with the cone at infinity equal to $Q^{-1}((-\infty,0])$ which is disjoint with $C_\infty X$.

To prove one of the inclusions needed for $m$ semicontinuity, take a sequence $D\ni x_\nu\to a$. Then for any sequence $y_\nu\in m(x_\nu)$ convergent to $y$ there is $$Q(y_\nu-x_\nu)\leq Q(y_\nu-p),\, \forall p\in X.$$
The inequality is preserved upon passing to the limit since $Q$ is continuous, thus $$\forall p\in X:\;Q(y-a)\leq Q(y-p),\text{ and }y\in m(a).$$

The opposite inclusion is first proved for $a\notin M_X$.
In such a case $m(a)=\lbrace x_0\rbrace$. Since $m$ has nonempty values and is locally bounded, we can pick a convergent sequence from any $m(a_\nu)$ where $a_\nu\to a$. Any such sequence must converge to $x_0$ due to the previous paragraph.

If $a\in M_X$ then from the first part we can approximate points of $m(b)$ for any $b\in (a,x)$ where $x\in m(a)$. Consequently $m(a)\subset \limsup m(x).$

\end{proof}

\begin{ex}
The condition concerning the Lipschitz constant in the previous theorem is optimal. Indeed, consider $$X=\lbrace (t,t^2/(1+t^2),t)\in\mathbb{R}_{(2,1)}\mid t\in\mathbb{R}\rbrace.$$ It is clearly a $1-$pseudo-Lipschitz set, yet $\rho(a)=-\infty$ and $m(a)$ is empty for any $a\notin V:=\lbrace (v,w,v)\mid w,v\in\mathbb{R}\rbrace$. In the same time for points $\boldsymbol{v}=(v,w,v)\in V$ there is $\rho(\boldsymbol{v})=0$ and $m(\boldsymbol{v})=(v,v^2/(1+v^2),v)$.  
 \end{ex}

\begin{cor}
Let $X$ be a $L-$pseudo-Lipschitz definable subset of $\mathbb{R}_{(l,p)}$ with $L< 1$, then the function $\rho$ is continuous.
\end{cor}
\begin{proof}
Choose $x\in\mathbb{R}_{(l,p)}$. Since the multifunction $m$ is upper-semi\-con\-ti\-nu\-ous, for any $V$ - the neighbourhood of $m(x)$ there exists $U$ - a neighbourhood of $x$ such that $m(U)\subset V$. Take $V$ relatively compact, then for any point $y\in U$ there is $\rho(y)=\inf\lbrace Q(y-a)\mid a\in X\cap V\rbrace $. Now 
$$\inf_{a\in X\cap V} \lbrace Q(y-a)\rbrace = Q(y-x)+\inf_{a\in X\cap V}\lbrace 2\langle\langle y-x,x-a\rangle\rangle +Q(x-a)\rbrace.$$
The term $Q(y-x)$ tends to zero when $y\to x$, the infimum on the other hand can be bounded from below by $$\inf_{a\in X\cap V}\lbrace 2\langle\langle y-x,x-a\rangle\rangle +Q(x-a)\rbrace \geq\inf_{a\in X\cap V}\lbrace2\langle\langle y-x,x-a\rangle\rangle\rbrace +\rho(x).$$ Since $V$ was relatively compact the norm of $x-a$ is bounded and $\langle\langle y-x,x-a\rangle\rangle$ tends to zero as $y\to x$, thus $\liminf \rho(y)\geq \rho(x)$. 
On the other hand  $\rho(y)\leq Q(y-b)$ for any $b\in m(x)$. Since $Q$ is continuous we obtain $\limsup \rho(y)\leq Q(x-b)=\rho(x).$
\end{proof}

\begin{thm}\label{gradient rho}
Let $X$ be a $L-$pseudo-Lipschitz definable subset of $\mathbb{R}_{(l,p)}$ with $L< 1$, then $\rho$ is locally lipschitz and differentiable for any point $x\in \mathbb{R}_{(l,p)}\backslash M_X$. Moreover for any such $x$ there is $\nabla \rho(x)=2A(x-m(x))$ thus $\rho$ is $\mathscr{C}^1$ smooth in $\mathbb{R}_{(l,p)}\backslash\overline{M_X}$
\end{thm}

\begin{proof}
For any $a\in \mathbb{R}_{(l,p)}$ the multifunction $m$ is locally bounded. It means that we can find an open Euclidean ball $B_a$ and $R>0$ such that for any $b\in B_a$ there is $m(b)\subset B(0,R)$. Then $\rho$ restricted to $B_a$ is given as an infimum of functions $Q_b(x):=Q(x-b)$ with $b\in B(0,R)\cap X$. Every $Q_b$ is Lipschitz in $B_a$ and the Lipschitz constants are uniformely bounded by certain $M>0$. Thus, $\rho$ is $M-$Lipschitz in $B_a$ as well.

Since $\rho$ is locally lipschitz, the Rademacher theorem asserts that it is differentiable almost everywhere. In o-minimal structures it translates to the differentiability on an open and dense subset of $\mathbb{R}_{(l,p)}.$

For $x\notin M_X$ we investigate now the quotient $$(\rho(x+h)-\rho(x)-2(x-m(x))^TAh)/\|h\|=$$
$$(Q(x+h-m(x+h))-Q(x-m(x))-2\langle\langle x-m(x),h\rangle\rangle)/\|h\|=$$
$$(Q(m(x+h))-Q(m(x))-2\langle\langle x+h,m(x+h)-m(x)\rangle\rangle+o(\|h\|))/\|h\|.$$

Where the second equality is due to $Q(x+h)-Q(x)-2\langle\langle x,h\rangle\rangle=o(\|h\|).$
We can further simplify the last line, down to
\begin{equation}
\label{eqn:star}\langle\langle m(x+h)+m(x)-2x-2h,\frac{m(x+h)-m(x)}{\|h\|}\rangle\rangle.\end{equation}
Now, since $m$ is definable it is differentiable in an open and dense subset of $\mathbb{R}_{(l,p)}$. Therefore, for almost all $x$ we can write \begin{equation}\label{eqn:star2}\frac{m(x+h)-m(x)}{\|h\|}=\langle\nabla m(x),h/\|h\|\rangle +o(\|h\|).\end{equation}
Clearly, the quotient \ref{eqn:star2} is bounded and implies the boundedness of the expression $\ref{eqn:star}$.

Moreover $m(x+h)\in X$ and $m(x+h)\to m(x)$ when $h\to 0$. It means, that any accumulation point of $\frac{m(x+h)-m(x)}{\|h\|}$ ( $h\to 0$) is an element of the tangent cone $C_{m(x)}X$. For any sequence $h\to 0$ yeilding a convergent expression $\ref{eqn:star}$ we obtain therefore a limit $$2\langle\langle m(x)-x,v\rangle\rangle,\, v\in C_{m(x)}X$$
which has to equal zero. Indeed, a lower estimate is given by 
Proposition~\ref{properties of N}. On the other hand, should $\langle\langle m(x)-x,v\rangle\rangle$ become positive, for $h$ small enough we would obtain $Q(x+h-m(x+h))>Q(x+h-m(x))$ which is a contradiction.

For an open and dense subset of $\mathbb{R}_{(l,p)}$ we have then $$\nabla \rho(x)=2A(x-m(x)).$$ The general theory of Clarke's subdifferential and theorem~\ref{continuity of m} assert now that $\rho$ is differentiable whenever $m(x)$ is a singleton. Moreover $\rho$ is $\mathscr{C}^1$ smooth outside the closure of $M_X$.

\end{proof}

\begin{cor}
For any $x\notin M_X$ there is $m(x)=x-A\nabla\rho/2.$
\end{cor}
\begin{proof}
It simply suffices to observe that $A^2=I$. Then the previous theorem gives the assertion.
\end{proof}

\begin{thm}
Let $X$ be a $L-$pseudo-Lipschitz definable subset of $\mathbb{R}_{(l,p)}$ with $L< 1$ then
\begin{enumerate}
    \item $\partial \rho(x)=cnv\lbrace 2A(x-y)\mid y\in m(x)\rbrace=\lbrace 2A(x-y)\mid y\in cnv\,m(x)\rbrace$
    \item $x\in M_X\iff \#\partial \rho (x)\neq 1\iff x\notin D_\rho$.
    \item $\nabla \rho=0\iff x\in X$
    \item $\nabla \rho$ is continuous in $\mathbb{R}_{(l,p)}\backslash M_X$
\end{enumerate}
\end{thm}

\begin{que}
Is it true that $M_X\neq \emptyset$ if $\pi_l(X)$ do not contain any half-space?
\end{que}

\begin{lem}
Let $X \subset \mathbb{R}^n$ be a closed, $L-$pseudo-Lipschitz set, $x_0 \in\mathbb{ R}^n \backslash X$ a point and $B = B(x_0 , r)$ a pseudo-ball such that $B \cap X = \bigcup ^k_{j=1} X_j$ where the sets $X_j$ are nonempty and pairwise disjoint, and for at least one $i$,
$X_i \cap B \neq \emptyset$. Then there exists a neighbourhood $U$ of $x_0$ such that $$\rho(x,X)=\min_{j=1}^k \rho(x,X_j),\, x\in U$$
\end{lem}

\section{Central set}
The definition of a central set in the Euclidean space $\mathbb{E}$ is as follows.
Firstly we define a collection of open balls included in the complement of $X$
$$\mathcal{B}:=\lbrace B(x,r)\mid  B(x,r)\subset \mathbb{E}\backslash X\rbrace.$$
Then we call the ball $B\in\mathcal{B}$ maximal if 
$$\forall B'\in\mathcal{B}: \, B\subset B'\Rightarrow B=B'.$$
The central set $C_X$ of $X$ collects the centres of all maximal balls in $\mathcal{B}$.

While adopting the notion in the pseudo-Euclidean setting one stumbles upon a problem. If $Q$ is not positive definite the inclusion between any two balls $B,B'$ implies their common centre. Therefore the central set $C_X$ equals the whole complement of $X$. Indeed, every point $x$ of $X$ complement is a centre of a maximal ball $B(x,\rho(x))$. The definition fails to yield any valuable study.

Hence a different, more local, approach is needed. As a source of our motivation we pick an observation that for any maximal ball $B$ of centre $c$ and $x\in \overline{B}\cap X\subset \mathbb{E}$ there is $$\sup\lbrace t\geq 0\mid x\in m(x+t(c-x)/\|c-x\|)\rbrace=\|c-x\|.$$ We construct the central set in the pseudo-Euclidean spaces by collecting all the points in the form $x+r_vv$ where $x\in X,v\in \mathcal{N}_xX$ and $r_v=\sup\lbrace t\geq 0\mid x\in m(x+tv)\rbrace$. The main point in the definition is to establish a suitable object for the following Theorem which is crucial while investigating the limit passing of the medial axis families.
\begin{thm}\label{centralny_zawiera_szkielet}
Let $X$ be a closed subset of pseudo-Euclidean space.
Then $M_X\subset C_X\subset \overline{M_X}$.
\end{thm}

\begin{proof}
The first inclusion is a consequence of Proposition~\ref{properties of N}(4). The proof of the second one is almost the same as in the standard Euclidean case (cf. \cite{Fremlin}). Take a point $a\in C_X$ and assume that $a$ is separated from the medial axis. Then there exists a radius $r>0$ such that the multifunction $m$ is univalued on $\mathbb{B}(a,r)$ and the image $m(\mathbb{B}(a,r))$ is a subset of $\mathbb{B}(m(a),\|m(a)-a\|)$. In such a case, due to Proposition\ref{properties of N}, $a\in C_X$ implies $a\neq tx+(1-t)m(x)$ for any $x\in\mathbb{B}(a,r)$. The function $H:(t,x)\rightarrow tx+(1-t)m(x)$ is therefore a homothopy. Now, its composition with the spherical projection with the centre at $a$ is a homothopy that transports the sphere $\mathbb{S}(a,r)$ onto its proper subset which is a contradiction.
\end{proof}

\begin{que}
Is the definition of the Central set equivalent to the inclusion of a germ of the ball near the intersection points with $X$?
\end{que}

\section{Nash Lemma \& Kuratowski Limits}
The following theorem is a generalisation of a famous Nash lemma \cite{Nash} to the pseudo-Euclidean spaces.
\begin{thm}
Let $X \subset \mathbb{R}^n$ be a closed, $k-$dimensional $L-$pseudo-Lipschitz set with $L<1$, then $\overline{M_X}\cap Reg_2 X=\emptyset$
\end{thm}
\begin{proof}
Take $x\in Reg_2 X$ and a local $\mathscr{C}^2$ parametrisation of $X$, a function $g:V\to X\cap U,\,g(0)=x$.
Regard function $$F:U\times V\ni(u,t)\to(\langle\langle u-g(t),\frac{\partial g}{\partial t_i}(t)\rangle\rangle)_{i=1}^k\in\mathbb{R}^k$$

Since the partial derivatives of $g$ span $T_{g(t)}X$ we have $F(u,t)=0$ iff $u-g(t)\in \mathcal{N}_{g(t)}X$.
Moreover it is possible to pick a neighbourhood $U'$ of the point $x$ that $m(x')\in U\cap X$ for all $x'\in U'$. In other words $\forall x'\in U'\exists t\in V:g(t)\in m(x')$.

Let us compute now 
$$\det \frac{\partial F}{\partial t}(x,0)=-\det \left[\langle\langle \frac{\partial g}{\partial t_i}(0),\frac{\partial g}{\partial t_j}(0)\rangle\rangle \right]_{i,j=1}^k.$$
and note that by a linear change of variables in $V$ we can assure that 
$$\langle\langle \frac{\partial g}{\partial t_i}(0),\frac{\partial g}{\partial t_j}(0)\rangle\rangle=0\text{ for } i\neq j.$$ 

In the end, the determinant will be equal to $-\sum Q( \frac{\partial g}{\partial t_i}(0))$. Since all the partial derivatives are vectors of the tangent cone $T_{x}X$, the values of $Q( \frac{\partial g}{\partial t_i}(0))$ are positive. Therefore the sum cannot vanish.

Now, the implicit function theorem allows us to uniquely solve the equation $F(u,t)=0$ with respect to the variable $t$ in a certain neighbourhood of $(x,0)$. In details, we obtain a neighbourhood $W\times T\subset U\times V$ and  a function $\tau:W\to V$ such that $F(x,t)=0$ iff $t=\tau(x)$. Shrinking $U'$ again to enforce $m(U')\subset g(T)$ we obtain therefore by setting $g(\tau(u))$ a unique vector satisfying $u-g(\tau(u))\in \mathcal{N}_{g(\tau(u))}X$ for any given $u\in U'$. Thus the multifunction $m$ is univalued in $U'$. Consequently $M_X$ is disjoint with $U'$ thus it is separated from $Reg_2 X$.
\end{proof}

An immediate corollary of the theorem forces us to search for the intersection points solely in the $Sng_2 X$. The last set splits naturally into $Sng_1X$ and $Reg_1X\cap Sng_2X$. 


\begin{que}
Do the psuedo-Euclidean skeletons reach the set $X$ under the same conditions as when treated as a subset of Euclidan space?
\end{que}

\begin{lem}\label{ciagłość rho z parametrem}

Let $\lbrace X_t\rbrace_{t\in T}$ be a family of closed pseudo-L-lipschitz sets with $X_t\xrightarrow{K} X_0$, $L<1$ and $0\in\overline{T\backslash\lbrace 0\rbrace}$. Posit $\rho_t(x):=\rho(x,X_t)$ and $m_t(x)=m(x,X_t)$. Then $\lim_{(t,x)\to(0,a)}\rho_t(x)=\rho_0(a)$ and $\limsup_{t\to 0}m_t(a)\subset m_0(a).$
\end{lem}
\begin{proof}
Take any $y\in m(a)$, then for $t$ in a certain neighbourhood of $0$ there is $X_t\cap \mathbb{B}(y,1)\neq \emptyset$ due to the Kuratowski convergence of $X_t$. Since sets $X_t$ are pseudo-L-lipschitz, the inclusion occurs $$X_t\subset \bigcup_{x\in\mathbb{B}(y,1)}S_x\text{, where } S_x=\lbrace x'\mid \, \|x'-x\|_l>L\|x'-x\|_p\rbrace.$$
Thus, following the argument similar to the one of \ref{gradient rho}, all $\rho_t$ are {L-Lipschitz} with an universal constant $L_0$  in a neighbourhood of $a$. Thus, for $x_1,x_2$ close to $a$ there is $$|\rho_t(x_1)-\rho_t(x_2)|\leq L_0\|x_1-x_2\|.$$
In the same time for small $\varepsilon>0$ we have $$B(a,\rho_0(a)+\varepsilon)\cap X_0\neq \emptyset\text{ and }\overline{B(a,\rho_0(a)-\varepsilon)}\cap X_0=\emptyset.$$
Thus for $t$ in suitably small neighbourhood of $0$ there is $\rho_0(a)-\varepsilon<\rho_t(a)<\rho_0(a)+\varepsilon$. In the end we obtain thus $$|\rho_t(x)-\rho_0(a)|\leq |\rho_t(x)-\rho_t(a)|+|\rho_t(a)-\rho_0(a)|\leq L_0\|x-a\|+\varepsilon$$ which proves the convergence.

The inner continuity of $m_t$ follows, as for any point $x\in X_0\backslash m_0(a)$ by definition there is $Q(x-a)>\rho_0(a)$. However for a sequence of points $x_t\in X_t$ converging to $ x\in\limsup m_t(a)\backslash m_0(a)$ that would mean $$\rho_0(a)=\lim_{t\to 0}\rho_t(a)=\lim_{t\to 0}Q(x_t-a)= Q(x-a)>\rho_0(a)$$ which forms a clear contradiction. 
\end{proof}

\begin{thm}
Let $\lbrace X_t\rbrace_{t\in T}$ be a family of closed pseudo-L-lipschitz sets with $X_t\xrightarrow{K} X_0$, $L<1$ and $0\in\overline{T\backslash\lbrace 0\rbrace}$. Then $$\liminf_{t\to 0} M_t\supset M_0,$$
where $M_t:=M_{X_t}.$
\end{thm}
\begin{proof}
Suppose that there exists a point $$a\in M_0\backslash\liminf_{t\to 0}M_t.$$
It implies existence of a sequence $T\ni t_\nu\to0$ and an euclidean ball $B$ centred at $a$ such that $$B\cap M_{t_\nu}=\emptyset\text{, for all }\nu\in\mathbb{N}.$$ By shrinking the parameter space we can therefore assume that the intersection is empty for all $t\in T$. Mind that it does not change the convergence $X_t\to X_0$ in the induced topology.

For every $t\in T$ let $m_t$ denote the closest points multifunction for $X_t$. Since $a$ has a single closest point in $X_t$, it is possible to define $$r(t):=\sup\lbrace s\geq 0\mid m_t(a)\in m_t(a(s,t))\rbrace$$ for every $t\in T$, where $a(s,t):=sa+(1-s)m_t(a)$.

Note that since $a(1,t)=a$, the supremum is always greater or equal to one. Moreover, the point $a(r(t),t)$ belongs to the central set of $X_t$ whenever $r(t)<\infty$.

Our aim now is to prove that $\lim_{t\to 0}r(t)=1$. Since $C_X\subset\overline{M_X}$ such limit would contradict $B\cap M_t=\emptyset$ for small $t$. Clearly $\liminf r(t)\geq 1$ since $r(t)$ is at least equal to one. Suppose then that $\limsup r(t)>1$. In such a case it is possible to pick $c>0$ and a sequence $t_\nu\to 0$ for which the pseudoballs $B_\nu:=B(a(1+c,t_\nu),Q(a(1+c,t_\nu)-m_{t_\nu}(a))$ are disjoint with $X_{t_\nu}$. This sequence of balls converges to the closure of a ball $B_0:=B(a(1+c,0),Q(a(1+c,0)-y),$ where $y$ is a certain point in $m_0(a)$ (cf Lemma~\ref{ciagłość rho z parametrem}).
Now, should $B_0\cap X_0=\emptyset$, the point $y$ would be the closest to $a(1+c,0)$ and $\#m(a)=1$ since $a\in (y,a(1+c,0)$ which contradicts $a\in M_0$. On the other hand any point $x\in B_0\cap X_0$ would contradict the convergence of $X_t\to X_0$. 

Therefore, the set $M_0\backslash \liminf M_t$ ought to be empty.

\end{proof}

\section{Medial axis for hypersurfaces}
Observe firstly that under the restrictions introduced so far, the only scenarios for $X$ to be a hypersurface would be to consider $\mathbb{R}_{(l,p)}$ with $p$ equal zero or one. The first instance is precisely the Euclidean case. For $p=1$ we have what follows.

\begin{thm}
Assume that $X$ is $k-$codimensional $1-$pseudo-Lipschitz closed subset of $\mathbb{R}_{(l,p)}$ and $a\in X$. Let $x_1,\ldots, x_{k+1}$ be points in $ \mathbb{R}^n$  with $a\in m(x_i)$ for $i=1\ldots k+1$ such that the vectors $x_i-a$ are linearly
independent. Then $a\in Sng_1X$.
\end{thm}
\begin{proof}
Define the hyperplanes $H_i$ by $$H_i:=\lbrace v\in\mathbb{R}_{(l,p)}\mid \,\langle\langle x_i-a,v\rangle\rangle=0\rbrace.$$ If we had $a\in Reg_1 X$ then by the choice of $a$ we would have $x_i-a\in \mathcal{N}_aX$ for $i=1,\ldots,k+1$, and $T_aX\subset \bigcap_{i=1}^{k+1} H_i$. However, the intersection of $H_i$ is transversal, thus its codimension equals $k+1$  which is a contradiction. 
\end{proof}

\begin{thm}
Take $a\in X$ and assume that for all $x$ in some neighbourhood $U$ of $x_0\in \mathbb{R}^n$ we have $a\in m(x)$. Then $a\in Sng_1 X$.
\end{thm}
\begin{proof}
Simply apply the last theorem to any $\text{codim} X+1$ linearly independent vectors $x_i-a\in U-a$.
\end{proof}

\begin{thm}
Let $X$ be a $1-$pseudo-Lipschitz closed hypersurface of $\mathbb{R}_{(l,p)}$ and $a\in\mathbb{R}_{(l,p)}\backslash M_X$ be such that $m(a)\in Reg_1 X$. Then for $S=B\cap\Gamma$ where $\Gamma$ is any $\mathscr{C}^1$ submanifold with $T_a\Gamma$ transversal to $\mathbb{R}(a-m(a))$ and $B$ is sufficiently small open ball centred at $a$, the mapping $m|S:S\to m(S)$ is a homeomorphism onto the open subset $m(S)\subset Reg_1 X$. 
\end{thm}
\begin{proof}
By assumption, $m$ is univaluead in a neighbourhood $U$ of $a$, and thus continuous. Therefore for an open $V$ such that $V\cap X\subset Reg_1 X$ we can find an open ball $B$ centered at $a$ with $m(B)\subset V$. By the previous proposition for ant two points $x_1,x_2\in S$ we have an inclusion between segments $[x_1,m(a)],[x_2,m(a)]$ which (provided $B$ is small enough) means that $x_1=x_2$. Thus, $m|S$ is injective, and so by Brouwer Domain Invariance Theorem it is a homeomorphism onto an open subset $m(S)\subset Reg_1 X$ 
\end{proof}

\begin{rem}
The previous theorem includes, in particular, the case when $\Gamma$ equals the sphere $\mathbb{S}(m(a),\|a-m(a)\|)$.
\end{rem}

\printbibliography 
\end{document}